\documentclass{article}
\usepackage{graphicx} 

\usepackage{amsmath}
\usepackage{amsthm}
\usepackage{amssymb}

\usepackage{authblk}

\usepackage{color}

\renewcommand{\tt}[1]{\normalfont\texttt{#1}}

\newcommand{\rev}[1]{\overline{#1}}

\newcommand{\Gam}{\mathbf{\Gamma}}
\newcommand{\Bet}{\mathbf{B}}

\newtheorem{theorem}{Theorem}[section]
\newtheorem{lemma}[theorem]{Lemma}
\newtheorem{proposition}[theorem]{Proposition}

\theoremstyle{definition}

\newtheorem{problem}[theorem]{Problem}

\usepackage{tikz}
\usepackage[colorlinks=true,allcolors=black]{hyperref}
\usepackage{xurl}

\title{Avoiding abelian and additive powers in rich words}
\author{Jonathan Andrade}
\affil{\small{Department of Mathematics and Statistics\\
University of Victoria, Victoria, BC, Canada\\
\texttt{\href{mailto:jonathanandrade@uvic.ca}{jonathanandrade@uvic.ca}}}}
\author{Lucas Mol\footnote{Research of Lucas Mol is supported by NSERC grant RGPIN-2021-04084.}}
\affil{\small{Department of Mathematics and Statistics\\
Thompson Rivers University, Kamloops, BC, Canada\\
\texttt{\href{mailto:lmol@tru.ca}{lmol@tru.ca}}}}
\date{February 2025}

\begin{document}

\maketitle

\begin{abstract}
    This paper concerns the avoidability of abelian and additive powers in infinite rich words.  In particular, we construct an infinite additive $5$-power-free rich word over $\{\tt{0},\tt{1}\}$ and an infinite additive $4$-power-free rich word over $\{\tt{0},\tt{1},\tt{2}\}$.  The alphabet sizes are as small as possible in both cases, even for abelian powers.
\end{abstract}

\section{Introduction}
\label{Section:Introduction}

This paper concerns the avoidability of certain types of repetitions in words.  
We consider words over alphabets where it makes sense to sum the letters.  Over such an alphabet, the \emph{sum} of a word $w$ is simply the sum of the letters that make up $w$.  We focus on alphabets that are subsets of the integers with the usual sum.

We consider repetitions of three different forms.
For an integer $k\geq 2$, a word $w$ that can be written in the form $w=w_1w_2\cdots w_{k}$, where
\begin{itemize}
    \item $w_i$ is equal to $w_1$ for all $i\in\{2,\ldots,k\}$ is called an \emph{ordinary $k$-power};
    \item $w_i$ is an anagram of $w_1$ (i.e., $w_i$ can be obtained from $w_1$ by permuting the order of the letters) for all $i\in \{2,\ldots,k\}$ is called an \emph{abelian $k$-power};    
    \item $w_i$ has the same length and the same sum as $w_1$ for all $i\in\{2,\ldots,k\}$ is called an \emph{additive $k$-power}.
\end{itemize}
We often say \emph{square} instead of $2$-power, and \emph{cube} instead of $3$-power.

A word is \emph{ordinary} (resp.\ \emph{abelian}, \emph{additive}) \emph{$k$-power-free} if it contains no ordinary  (resp.\ abelian, additive) $k$-power as a factor.
Note that every ordinary $k$-power is an abelian $k$-power, and that every abelian $k$-power is an additive $k$-power.  
It follows that every additive $k$-power-free word is abelian $k$-power-free, and that every abelian $k$-power-free word is ordinary $k$-power-free.

Our motivating problem is the following.  Given a finite alphabet $X\subseteq \mathbb{Z}$, is there an infinite word over $X$ that is ordinary (resp.~abelian, additive) $k$-power-free?  
If so, then we say that ordinary (resp.~abelian, additive) $k$-powers are \emph{avoidable} over $X$.
If $|X|<2$, then it is easy to see that there is no such infinite word, so we may assume that $|X|\geq 2$.

Thue~\cite{Thue1912} solved this problem completely for ordinary $k$-powers by showing that there is an infinite ordinary cube-free binary word, and an infinite ordinary square-free ternary word.  (It is easy to show that there are only finitely many ordinary square-free words over a binary alphabet.)  In fact, Thue showed that there is a binary word that contains no \emph{overlap}, i.e., a word of the form $axaxa$, where $a$ is a single letter and $x$ is a (possibly empty) word.  Thue's results are considered the origin of the area of \emph{combinatorics on words}.

Erd\H{o}s~\cite{Erdos1961} was first to ask about the avoidability of abelian $k$-powers, and this variant of the problem has since been solved completely. 
Dekking~\cite{Dekking1979} constructed an infinite abelian $4$-power-free binary word, and an infinite abelian cube-free ternary word.  Ker\"anen~\cite{Keranen1992} constructed an infinite abelian square-free word over a four-letter alphabet. In all three cases, the alphabet size is as small as possible.

Finally, we summarize the progress on the problem for additive $k$-powers.  First note that over an integer alphabet of size two, words of the same length have the same sum if and only if they are anagrams of one another.  Thus, over an integer alphabet of size $2$, the notion of additive $k$-power is the same as the notion of abelian $k$-power.  Therefore, it follows from Dekking's result mentioned above that there is an infinite additive $4$-power-free word over every subset of $\mathbb{Z}$ of size at least two.
Cassaigne, Currie, Schaeffer, and Shallit~\cite{CassaigneCurrieSchaefferShallit2014} were the first to construct an infinite additive cube-free word over a finite subset of $\mathbb{Z}$ (namely $\{\tt{0},\tt{1},\tt{3},\tt{4}\}$).  Lietard and Rosenfeld~\cite{LietardRosenfeld2020} extended this result by showing that there is an infinite additive cube-free word over \emph{every} subset of $\mathbb{Z}$ of size $4$, except possibly those subsets equivalent to $\{\tt{0},\tt{1},\tt{2},\tt{3}\}$, i.e., subsets that are arithmetic progressions of length $4$.  It is still unknown whether there is an infinite additive cube-free word over $\{\tt{0},\tt{1},\tt{2},\tt{3}\}$.  Rao~\cite{Rao2015} constructed infinite additive cube-free words over several integer alphabets of size $3$. 
It is still unknown whether there exists an infinite additive square-free word over a finite subset of $\mathbb{Z}$.  However, Rao and Rosenfeld~\cite{RaoRosenfeld2018} constructed an infinite additive square-free word over a finite subset of $\mathbb{Z}^2$.

In this paper, we consider the avoidability of powers in \emph{rich} words. A word of length $n$ is called \emph{rich} if it contains $n+1$ distinct palindromes as factors, and an infinite word is called \emph{rich} if all of its finite factors are rich.  
Since their (implicit) introduction by Droubay, Justin, and Pirillo~\cite{DroubayJustinPirillo2001}, the language of rich words has been well-studied.
It is known that the set of infinite rich words contains several highly structured classes of words, including Sturmian words, episturmian words, and complementary symmetric Rote words (see~\cite{BlondinMasseEtAl2011,DroubayJustinPirillo2001}).

There are several papers concerning the avoidability of ordinary powers in rich words.  Pelantov\'{a} and Starosta~\cite{PelantovaStarosta2013} proved that every infinite rich word, over \emph{any} finite alphabet, contains an ordinary square. (In fact, they proved that every infinite rich word contains infinitely many distinct overlaps as factors.)  
On the other hand, it is known that there are infinite ordinary cube-free binary rich words; in fact, the results of Baranwal and Shallit~\cite{BaranwalShallit2019} and Currie, Mol, and Rampersad~\cite{CurrieMolRampersad2020} describe exactly which \emph{fractional} ordinary powers (between squares and cubes) are avoidable in binary rich words.

In this paper, we establish two results about the avoidability of abelian and additive powers in rich words.

\begin{theorem}\label{Theorem:Binary5PowerFree}
    There is an infinite additive $5$-power-free rich word over $\{\tt{0},\tt{1}\}$.
\end{theorem}

\begin{theorem}\label{Theorem:Ternary4PowerFree}
    There is an infinite additive $4$-power-free rich word over $\{\tt{0},\tt{1},\tt{2}\}$.  
\end{theorem}

These results are best possible in the sense that there are only finitely many abelian $4$-power-free rich words over any binary alphabet, and only finitely many abelian cube-free rich words over any ternary alphabet.  Using a backtracking algorithm, we determined that the longest abelian $4$-power-free binary rich word has length $2411$, and that the longest abelian cube-free ternary rich word has length $180$. These long backtracking searches were made possible by the use of the data structure \tt{eerTree} of Rubinchik and Shur~\cite{RubinchikShur2018}.

The constructions that we use to prove Theorem~\ref{Theorem:Binary5PowerFree} and Theorem~\ref{Theorem:Ternary4PowerFree} are obtained by iterating morphisms.  Let $\beta:\{\tt{0},\tt{1}\}^{*} \rightarrow \{\tt{0},\tt{1}\}^{*}$ and $\gamma:\{\tt{0},\tt{1},\tt{2}\}^*\rightarrow\{\tt{0},\tt{1},\tt{2}\}^*$ be the morphisms defined as follows:
    \begin{align*}
        \beta(\tt{0}) &= \tt{00001} & \gamma(\tt{0}) &= \tt{2} \\
        \beta(\tt{1}) &= \tt{01101} & \gamma(\tt{1}) &= \tt{101}\\
         & &                          \gamma(\tt{2}) &= \tt{10001}
    \end{align*}
Let $\Bet=\beta^\omega(\tt{0})$ and $\Gam=\gamma^\omega(\tt{1})$.
We prove Theorem~\ref{Theorem:Binary5PowerFree} by showing that $\Bet$ is rich and additive $5$-power-free, and we prove Theorem~\ref{Theorem:Ternary4PowerFree} by showing that $\Gam$ is rich and additive $4$-power-free.  

In order to prove the additive power-freeness of $\Bet$ and $\Gam$, we implemented the algorithm of Currie, Mol, Rampersad, and Shallit~\cite{CurrieMolRampersadShallit2021}, which is capable of deciding whether words produced by iterating certain morphisms are additive $k$-power-free. Our implementation of the algorithm can be found at \url{https://github.com/lgmol/Additive-Powers-Decision-Algorithm}.

In order to prove that $\Bet$ and $\Gam$ are rich, we take two different approaches.  We prove that $\Bet$ is rich using the automatic theorem-proving software \tt{Walnut}.  We prove that $\Gam$ is rich more or less directly using a well-known characterization of rich words.  We note that the morphism $\gamma^2$ can be ``conjugated'' to a morphism in the class $P_{\text{ret}}$, which has been studied by several authors in the context of rich words~\cite{BalkovaPelantovaStarosta2011,DolcePelantova2022}.  While there is a simple criterion for deciding whether a word obtained by iterating a \emph{binary} morphism in class $P_{\text{ret}}$ is rich~\cite{DolcePelantova2022}, no such simple criterion is known for ternary morphisms.  Ultimately, we found it slightly easier to work with $\gamma$ than the corresponding morphism in class $P_{\text{ret}}$.

The layout of the remainder of the paper is as follows.  In Section~\ref{Section:Background}, we provide the necessary background.  In Section~\ref{Section:Binary}, we prove Theorem~\ref{Theorem:Binary5PowerFree}.  In Section~\ref{Section:Ternary}, we prove Theorem~\ref{Theorem:Ternary4PowerFree}.  Finally, in Section~\ref{Section:Conclusion}, we discuss some open problems related to our work.

\section{Background}
\label{Section:Background}

We provide only a very brief overview of the basic background material on combinatorics on words.  See the books of Lothaire~\cite{Lothaire1983,Lothaire2002,Lothaire2005} or the more recent book of Shallit~\cite{Shallit2022} for a more comprehensive treatment. 

\subsection{Words}
\label{Subsection:Words}

An \emph{alphabet} is a nonempty finite set of symbols, which we refer to as \emph{letters}.  In this paper, we focus on alphabets that are finite subsets of the integers. 
Let $X$ be such an alphabet.  A \emph{word} over $X$ is a finite or infinite sequence of letters from $X$. We let $X^{*}$ denote the set of all finite words over the alphabet $X$.  The \emph{length} of a finite word $w$, denoted by $|w|$, is the number of letters that make up $w$. We let $\varepsilon$ denote the \emph{empty word}, which is the unique word of length $0$.  The \emph{sum} of a finite word $w$, denoted by $\sum w$, is the sum of the letters that make up $w$.

For a finite word $x$ and a finite or infinite word $y$, the \emph{concatenation} of $x$ and $y$, denoted by $xy$, is the word consisting of all of the letters of $x$ followed by all of the letters of $y$.  Suppose that a finite or infinite word $w$ can be written in the form $w=xyz$, where $x$, $y$, and $z$ are possibly empty words.  Then the word $y$ is called a \emph{factor} of $w$, the word $x$ is called a \emph{prefix} of $w$, and the word $z$ is called a \emph{suffix} of $w$.  If $x$ and $z$ are nonempty, then $y$ is called an \emph{internal factor} of $w$.  If $yz$ is nonempty, then $x$ is called a \emph{proper prefix} of $w$, and if $xy$ is nonempty, then $z$ is called a \emph{proper suffix} of $w$.

\subsection{Morphisms}
\label{Subsection:Morphisms}

For alphabets $X$ and $Y$, a $\emph{morphism}$ from $X^{*}$ to $Y^{*}$ is a function $h: X^{*} \rightarrow Y^{*}$ that satisfies $h(uv) = h(u)h(v)$ for all words $u,v \in X^{*}$.  We say that $h$ is \emph{strictly growing} if $|h(a)|\geq 2$ for all $a\in X$.  For an integer $k \geq 1$, we say that $h$ is \emph{k-uniform} if $|h(a)| = k$ for all $a \in X$.  

Let $h:X^{*} \rightarrow X^{*}$ be a morphism. For all words $x\in X^*$, we define $h^{0}(x) = x$, and $h^n(x)=h(h^{n-1}(x))$ for all integers $n\geq 1$.  For a letter $a\in X$, we say that $h$ is \emph{prolongable} on $a$ if $h(a) = ax$ for some word $x \in X^{*}$ and $h^n(x)\neq \varepsilon$ for all $n\geq 0$. 
If $h$ is prolongable on $a$ with $h(a)=ax$, then it is easy to show that for every integer $n\geq 1$, we have
\[h^n(a) = axh(x)h^{2}(x)\cdots h^{n-1}(x).\]
Thus, each $h^n(a)$ is a prefix of the infinite word
\[
h^\omega(a)=axh(x)h^2(x)h^3(x)\cdots.
\]
Note that $h^\omega(a)$ is a fixed point of $h$, i.e., $h(h^\omega(a))=h^\omega(a)$, where the morphism $h$ is extended to infinite words in the natural way.

\subsection{Rich Words}
\label{Subsection:RichWords}

For a finite word $w=w_1w_2\cdots w_{n-1}w_n$, where the $w_i$ are letters, the \emph{reversal} of $w$, denoted $\rev{w}$, is the word obtained by writing the letters of $w$ in the opposite order, i.e.,
\[
\rev{w}=w_nw_{n-1}\cdots w_2w_1.
\]
A word $w$ is a \emph{palindrome} if $\rev{w}=w$, i.e., if $w$ reads the same backward as forward.  Some examples of palindromes in the English language include \tt{racecar}, \tt{level}, and \tt{kayak}. 

Droubay, Justin, and Pirillo~\cite{DroubayJustinPirillo2001} were first to observe that every word of length $n$ contains at most $n+1$ distinct palindromes (including the empty word) as factors.  This explains the eventual use of the term \emph{rich} for words of length $n$ containing $n+1$ distinct palindromes as factors; these words are ``rich'' in palindromes.

Let us briefly explain the idea behind Droubay, Justin, and Pirillo's observation.  A factor $u$ of a word $w$ is said to be \emph{unioccurrent} in $w$ if $u$ occurs exactly once in $w$. For example, in the word $\tt{011011}$, the factor $\tt{11011}$ is unioccurrent, while the factor $\tt{11}$ is not.  Observe that every finite word has at most one unioccurrent palindromic suffix (which must necessarily be its longest palindromic suffix). Further, for a finite word $w$ and a letter $a$, every palindromic factor of $wa$ that is not a factor of $w$ must be a unioccurrent palindromic suffix of $wa$.  So adding a letter to the end of a finite word can add at most one new palindromic factor to the word.  Since the empty word contains just one palindromic factor (namely $\varepsilon$ itself), it follows by a straightforward induction that every word of length $n$ contains at most $n+1$ distinct palindromic factors.

We will use the following basic results about rich words, which can be proven using simple variations on the idea from the previous paragraph.  These results were first proven by Droubay, Justin, and Pirillo~\cite{DroubayJustinPirillo2001}, but we state them using the updated terminology of Glen, Justin, Widmer, and Zamboni~\cite{GlenJustinWidmerZamboni2009}.

\begin{theorem}
\label{Theorem:FiniteUnioccurrent}
    A finite word $w$ is rich if and only if every prefix (resp.~suffix) of $w$ has a unioccurrent palindromic suffix (resp.~prefix).
\end{theorem}

\begin{lemma}
\label{Lemma:FactorsOfRichWordsAreRich}
    If a finite word $w$ is rich, then every factor of $w$ is rich.
\end{lemma}

\begin{theorem} 
\label{Theorem:InfiniteUnioccurrent}
    An infinite word $\mathbf{w}$ is rich if and only if every finite prefix of $\mathbf{w}$ has a unioccurrent palindromic suffix.
\end{theorem}

\subsection{A decision algorithm for additive \texorpdfstring{$k$}{k}-power-freeness}
\label{Subsection:DecisionAlgorithm}

In order to establish that $\Bet$ is additive $5$-power-free and $\Gam$ is additive $4$-power-free, we implemented the decision algorithm described by Currie, Mol, Rampersad, and Shallit~\cite{CurrieMolRampersadShallit2021}, which applies to fixed points of certain ``affine'' morphisms.  A morphism $f: X^{*} \rightarrow Y^{*}$ where $X,Y\subseteq \mathbb{Z}$ is called \emph{affine} if there exist $a,b,c,d\in\mathbb{Z}$ such that for all $x \in X$, we have $|f(x)| = a +bx$ and $\sum f(x) = c + dx$. In this case, we define the matrix
\[
M_{f} = 
\begin{bmatrix}
    a & b \\
    c & d
\end{bmatrix}.
\]

\begin{theorem}
[Currie, Mol, Rampersad, and Shallit~\cite{CurrieMolRampersadShallit2021}]\label{Theorem:DecisionAlgorithm}
    Let $f:X^{*} \rightarrow X^{*}$ be a strictly growing affine morphism such that $f$ is prolongable on $a$. If $M_{f}$ is invertible and every eigenvalue $\lambda$ of $M_{f}$ satisfies $|\lambda| > 1$, then it is possible to decide whether or not $f^{\omega}(a)$ is additive $k$-power-free.
\end{theorem}

Recall the morphisms $\beta$ and $\gamma$ defined in Section~\ref{Section:Introduction}, and let $\delta=\gamma^2$.  So we have
\begin{align*}
\beta(\tt{0}) &= \tt{00001} & \gamma(\tt{0})&=\tt{2} & \delta(\tt{0})&=\tt{10001}\\
\beta(\tt{1}) &= \tt{01101} & \gamma(\tt{1})&=\tt{101} & \delta(\tt{1})&=\tt{1012101}\\
& & \gamma(\tt{2})&=\tt{10001} & \delta(\tt{2})&=\tt{101222101}
\end{align*}
and $\Bet=\beta^\omega(\tt{0})$ and $\Gam=\gamma^\omega(\tt{1})=\delta^\omega(\tt{1})$.
First note that $\beta$ and $\delta$ are strictly growing, and that $\beta$ is prolongable on $\tt{0}$, while $\delta$ is prolongable on $\tt{1}$.  Further, for all $x\in\{\tt{0},\tt{1}\}$, we have $|\beta(x)|=5+0x$ and $\sum \beta(x)=1+2x$, so $\beta$ is affine, and we have 
\[
M_{\beta} = 
\begin{bmatrix}
    5 & 0 \\
    1 & 2
\end{bmatrix}.
\]
Similarly, for the morphism $\delta$, we have $|\delta(x)| = 5 + 2x$ and $\sum \delta(x) = 2 + 4x$ for all  $x\in \{\tt{0},\tt{1},\tt{2}\}$. Therefore, $\delta$ is affine, and we have 
\[
M_{\delta} = 
\begin{bmatrix}
    5 & 2 \\
    2 & 4
\end{bmatrix}.
\]
Note that $M_\beta$ is invertible with eigenvalues $2$ and $5$, and that $M_\delta$ is invertible with eigenvalues $\frac{9+\sqrt{17}}{2}$ and $\frac{9-\sqrt{17}}{2}$. Hence, the conditions of Theorem~\ref{Theorem:DecisionAlgorithm} are met, and we can use the decision algorithm of Currie, Mol, Rampersad, and Shallit to prove that $\Bet$ is additive $5$-power-free and $\Gam$ is additive $4$-power free.

The algorithm makes use of the \emph{template method}, which was originally developed by Currie and Rampersad~\cite{CurrieRampersad2012} to decide abelian $k$-power-freeness.  Rao and Rosenfeld~\cite{RaoRosenfeld2018} established a more general version of the method, and found a way to apply it to both abelian and additive powers. The algorithm which we implement has more restrictive conditions on the morphism and works only for additive powers, but it has the advantage of being easy to implement in general (without worrying about floating point error, for example) and relatively efficient compared to previous template method algorithms.

The essential idea behind the algorithm is that if $f^\omega(a)$ contains a sufficiently long additive $k$-power, then it must contain some short word that is not ``too far'' from an additive $k$-power.  We use objects called \emph{templates} to describe structures that are ``close to'' additive $k$-powers.  The conditions on $f$ ensure that if $f^\omega(a)$ contains a sufficiently long additive $k$-power, then it contains some short instance of one of only finitely many \emph{ancestor templates}.  So to prove that $f^\omega(a)$ is additive $k$-power-free, it suffices to complete the following three steps.
\begin{itemize}
    \item \textbf{Initial Check:} Check that there is no ``short'' additive $k$-power in $f^\omega(a)$.
    \item \textbf{Calculating Ancestors:} Find the set $\mathcal{A}$ of all ancestor templates.
    \item \textbf{Final Check:} Check that there is no ``short'' word in $f^\omega(a)$ that matches an ancestor template.
\end{itemize}
The exact length of the ``short'' factors that need to be checked depends on the integer $k$, the morphism $f$, and the set $\mathcal{A}$ of all ancestor templates.  For more details concerning the algorithm and its implementation, see the undergraduate Honour's thesis of the first author~\cite{Andrade2024}.

Finally, it is interesting to note that while the conditions on the morphism for the algorithm of Rao and Rosenfeld~\cite{RaoRosenfeld2018} seem less restrictive in general than those of the algorithm we implement, the moprhisms $\gamma$ and $\delta=\gamma^2$ do not satisfy the conditions of Rao and Rosenfeld.  In particular, the incidence matrix of the morphism $\gamma$, namely
\[
\begin{bmatrix}
    0 & 1 & 3\\
    0 & 2 & 2\\
    1 & 0 & 0
\end{bmatrix},
\]
has eigenvalue $1$.  Rao and Rosenfeld~\cite[Section 6.1]{RaoRosenfeld2018} stated that they did not know of such a morphism with an abelian $k$-power-free fixed point.

\section{An additive 5-power-free binary rich word}
\label{Section:Binary}

In this section, we prove Theorem~\ref{Theorem:Binary5PowerFree}.

\begin{proposition}\label{Proposition:Binary5PowerFree}
    The word $\Bet=\beta^\omega(\tt{0})$ is additive $5$-power-free.
\end{proposition}

\begin{proof}
    We verify that $\Bet$ is additive $5$-power-free by running the algorithm described in Section~\ref{Subsection:DecisionAlgorithm}.  Our implementation of the algorithm can be found at \url{https://github.com/lgmol/Additive-Powers-Decision-Algorithm}, and includes this specific run of the algorithm.
\end{proof}

\begin{proposition}\label{Proposition:BinaryRich}
    The word $\Bet$ is rich.
\end{proposition}
 
\begin{proof}
Since the morphism $\beta$ is $5$-uniform, the word $\Bet$ is $5$-automatic, and we can use the automatic theorem-proving software \tt{Walnut}, which is capable of deciding statements written in a certain first-order logic about automatic sequences.  See the recent book of Shallit~\cite{Shallit2022} for a comprehensive guide to using \tt{Walnut}.

By Theorem~\ref{Theorem:InfiniteUnioccurrent}, it suffices to show that every finite prefix of $\Bet$ has a unioccurrent palindromic suffix.  We express this property in the first-order logic of \tt{Walnut} below.  Similar \tt{Walnut} commands were used by Baranwal and Shallit~\cite{BaranwalShallit2019} and Schaeffer and Shallit~\cite{SchaefferShallit2016} in proving that other infinite words are rich.  

{\footnotesize
\begin{verbatim}
    morphism b "0->00001 1->01101":
        # defines the morphism beta (using the letter b as a name)
\end{verbatim}
\begin{verbatim}
    promote B b:
        # defines a DFAO generating the fixed point B of b
\end{verbatim}
\begin{verbatim}
    def FactorEq "?msd_5 Ak (k<n)=>(B[i+k]=B[j+k])":
        # takes 3 parameters i,j,n and returns true if 
        # the length-n factors of B starting at indices i and j are equal
\end{verbatim}
\begin{verbatim}
    def Occurs "?msd_5 (m<= n) & ( Ek (k+m<=n) & $FactorEq(i,j+k,m))":   
        # takes 4 parameters i,j,m,n and returns true if 
        # the length-m factor of B starting at index i 
        # occurs in the length-n factor of B starting at index j
\end{verbatim}
\begin{verbatim}   
    def Palindrome "?msd_5 Aj,k ((k<n) & (j+k+1=n)) => (B[i+k]=B[i+j])":
        # takes 2 parameters i,n and returns true if 
        # the length-n factor of B starting at index i is a palindrome
\end{verbatim}
\begin{verbatim}        
    def BisRich  "?msd_5 An Ej $Palindrome(j,n-j) & ~$Occurs(j,0,n-j,n-1)":
        # returns true if every finite prefix of B has a 
        # unioccurrent palindromic suffix, i.e., if B is rich
\end{verbatim}
}
\noindent
When we run all of these commands, \tt{Walnut} returns \tt{TRUE}. 
\end{proof}

Theorem~\ref{Theorem:Binary5PowerFree} follows immediately from Proposition~\ref{Proposition:Binary5PowerFree} and Proposition~\ref{Proposition:BinaryRich}.

\section{An additive 4-power-free ternary rich word}
\label{Section:Ternary}

In this section, we prove Theorem~\ref{Theorem:Ternary4PowerFree}.  

\begin{proposition}\label{Proposition:Ternary4PowerFree}
    The word $\Gam=\gamma^\omega(\tt{1})$ is additive $4$-power-free.
\end{proposition}

\begin{proof}
    We verify that $\Gam$ is additive $4$-power-free by running the algorithm described in Section~\ref{Subsection:DecisionAlgorithm} with the morphism $\delta=\gamma^2$.  Our implementation of the algorithm can be found at \url{https://github.com/lgmol/Additive-Powers-Decision-Algorithm}, and includes this specific run of the algorithm.
\end{proof}

In order to prove that $\Gam$ is rich, we start with the following basic lemma.

\begin{lemma}\label{Lemma:PalindromePreservation}
    If $u\in\{\tt{0},\tt{1},\tt{2}\}^*$ is a palindrome, then $\gamma(u)$ is a palindrome.
\end{lemma}

\begin{proof}
    We proceed by induction on the length of $u$.  First consider the case that $|u|\leq 1$.  The statement holds trivially when $u=\varepsilon$, and it is straightforward to check that $\gamma(\tt{0})$, $\gamma(\tt{1})$, and $\gamma(\tt{2})$ are palindromes.  Now suppose for some integer $n\geq 2$ that $u$ has length $n$, and that for every palindrome $v\in\{\tt{0},\tt{1},\tt{2}\}^*$ of length less than $n$, the word $\gamma(v)$ is a palindrome.  Since $|u|\geq 2$, we can write
    \[
        u=awa
    \]
    for some letter $a\in\{\tt{0},\tt{1},\tt{2}\}$ and some palindrome $w\in\{\tt{0},\tt{1},\tt{2}\}^*$.  By the inductive hypothesis, the words $\gamma(w)$ and $\gamma(a)$ are both palindromes.  Thus, we have
    \begin{align*}
        \rev{\gamma(u)}&=\rev{\gamma(a)}\ \rev{\gamma(w)}\ \rev{\gamma(a)}=\gamma(a) \gamma(w) \gamma(a)=\gamma(u),
    \end{align*}
    and we conclude that $\gamma(u)$ is a palindrome.
\end{proof}

\begin{proposition}\label{Proposition:TernaryRich}
    The word $\Gam$ is rich.
\end{proposition}

\begin{proof}
By Theorem~\ref{Theorem:InfiniteUnioccurrent}, it suffices to show that every finite prefix of $\Gam$ has a unioccurrent palindromic suffix.  We proceed by induction on the length of the prefix.  For the base case, we check by computer that every prefix of $\Gam$ of length at most $14$ has a unioccurrent palindromic suffix.

    Now suppose for some integer $n\geq 15$ that every prefix of $\Gam$ of length less than $n$ has a unioccurrent palindromic suffix, and let $P$ be the prefix of $\Gam$ of length $n$.  Let $p$ be the longest prefix of $\Gam$ such that $P$ has $\gamma(p)$ as a prefix, and let $a\in\{\tt{0,1,2}\}$ be the unique letter such that $q=pa$ is a prefix of $\Gam$.  Then we can write 
    \[
        P=\gamma(p)r,
    \]
where $r$ is a proper prefix of $\gamma(a)$.  (See Figure~\ref{Figure:PrefixP}.)  Note that since $n\geq 15$, we have $|p|\geq 5$.
    
Since $\gamma$ is nonerasing and $|\gamma(\tt{1})|=3$, we have 
    \[
    |q|=|p|+1<|\gamma(p)|\leq n.
    \]  
Hence, by the inductive hypothesis, every prefix of $q$ has a unioccurrent palindromic suffix.  It follows by Theorem~\ref{Theorem:FiniteUnioccurrent} that $p$ and $q$ are rich.  Let $s$ be the unioccurrent palindromic suffix of $p$, and let $t$ be the unioccurrent palindromic suffix of $q$.  Since $|p|\geq 5$ and every letter appears in the prefix of $p$ of length $4$, both $s$ and $t$ must have length at least $2$.

    \begin{figure}
        \centering
        \begin{tikzpicture}
            \draw (0,0) rectangle node{$P$} (8,1);
            \draw (0,1) rectangle node{$\gamma(p)$} (7,2);
            \draw (7,1) rectangle node{$r$} (8,2);
            \draw (0,2) rectangle node{$\gamma(p)$} (7,3);
            \draw (7,2) rectangle node{$\gamma(a)$} (9,3);
            \draw (0,3) rectangle node{$\gamma(q)$} (9,4);
            \draw (0,4) -- (0,5) -- (10,5);
            \draw (9,4) -- (10,4);
            \node[right] at (10,4) {$\cdots$};
            \node[right] at (10,5) {$\cdots$};
            \node at (5,4.5) {$\Gam$};
        \end{tikzpicture}
        \caption{The prefix $P$ of $\Gam$}
        \label{Figure:PrefixP}
    \end{figure}
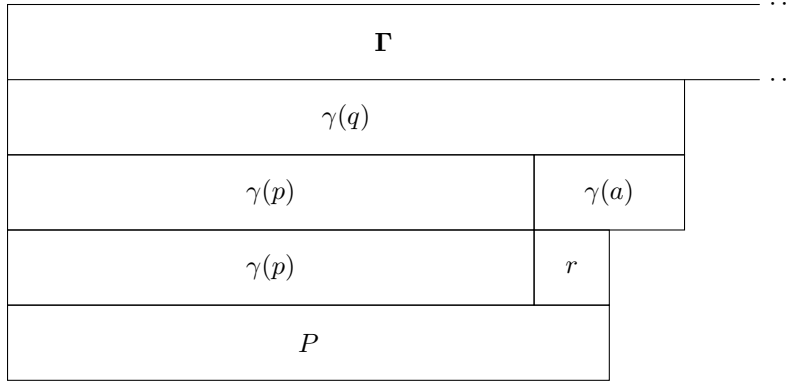

    \smallskip

    \noindent
    \textbf{Case 1:} Suppose that $r=\varepsilon$.  

    \smallskip
    
    \noindent
    We claim that 
    \[
        S=\gamma(s)
    \]
    is a unioccurrent palindromic suffix of $P$.  Since $s$ is a palindrome, it follows from Lemma~\ref{Lemma:PalindromePreservation} that $S$ is a palindrome.  Note that every occurrence of $S$ in $P$ corresponds to an occurrence of $s$ in $p$.  Thus, since $s$ is unioccurrent in $p$, we conclude that $S$ is unioccurrent in $P$.

    \smallskip
    
    \noindent
    \textbf{Case 2:} Suppose that $r\in\{\tt{100},\tt{1000}\}$.
    
	\smallskip  
    
    \noindent
    Since $r$ is only a prefix of the $\gamma$-image of the letter $\tt{2}$, we must have $a=\tt{2}$.  Write $t=\tt{2}t'\tt{2}$.  We claim that 
    \[
        S=\rev{r}\gamma(t')r
    \]
    is a unioccurrent palindromic suffix of $P$.  Since $t$ is a palindrome, the word $t'$ is also a palindrome, and it follows easily from Lemma~\ref{Lemma:PalindromePreservation} that $S$ is a palindrome.  Note that $\rev{r}$ is a suffix of $\gamma(\tt{2})$, so that $S$ is indeed a suffix of $P$.  Finally, since $r$ (resp.~$\rev{r}$) is only a prefix (resp.~suffix) of the $\gamma$-image of $\tt{2}$, we see that every occurrence of $S$ in $P$ corresponds to an occurrence of $t=\tt{2}t'\tt{2}$ in $q$.  Since $t$ is unioccurrent in $q$, we conclude that $S$ is unioccurrent in $P$.
    
%
   	
   	\smallskip

    \noindent 
    \textbf{Case 3:} Suppose that $r\in\{\tt{1},\tt{10}\}$.

    \smallskip
    \noindent
    Since $r$ is not a prefix of $\gamma(\tt{0})$, we must have $a\in\{\tt{1},\tt{2}\}$.  We have two subcases.

	\smallskip	
	
    \noindent
    \textbf{Case 3.1:} Suppose that $p$ ends in $\tt{0}$.

    \smallskip

    \noindent
	Since $\tt{02}$ is not a factor of $\Gam$, we must have $a=\tt{1}$.  Write $t=\tt{1}t'\tt{1}$.  We claim that
	\[
	S=\overline{r}\gamma(t')r
	\]  
	is a unioccurrent palindromic suffix of $P$.  The proof that $S$ is a palindromic suffix of $P$ can be completed as in Case 2.  Since $p$ ends in $\tt{0}$ and $t'$ is a palindromic suffix of $p$, we see that $t'$ begins and ends in $\tt{0}$.  Since $\tt{20}$ and $\tt{02}$ are not factors of $\Gam$, we see that every occurrence of $S$ in $P$ corresponds to an occurrence of $t=\tt{1}t'\tt{1}$ in $q$.  Since $t$ is unioccurrent in $q$, we conclude that $S$ is unioccurrent in $P$.
    
    \smallskip

    \noindent
    \textbf{Case 3.2:} Suppose that $p$ ends in $\tt{1}$ or $\tt{2}$.

    \smallskip

    \noindent 
    Write $p=\tt{1}p'$.  Since $p$ is rich, we see by Lemma~\ref{Lemma:FactorsOfRichWordsAreRich} that $p'$ is also rich.  Let $s'$ be the unioccurrent palindromic suffix of $p'$.  (Note that $s'$ only differs from $s$ in the case that $s=p$.)  Since $p'$ has prefix $\tt{0121}$, we must have $|s'|\geq 2$.
    
    We claim that 
    \[
        S=\rev{r}\gamma(s')r
    \]
    is a unioccurrent palindromic suffix of $P$.  Since $s'$ is a palindrome, we see by Lemma~\ref{Lemma:PalindromePreservation} that $\gamma(s')$ is a palindrome, and in turn that $S$ is a palindrome.  
    
    Next, we show that $S$ is a suffix of $P$.  First note that since $s'$ is a proper suffix of $p$, there is some letter $b$ such that $bs'$ is a suffix of $p$, and in turn $bs'a$ is a suffix of $q$. Suppose towards a contradiction that $b=\tt{0}$.  Then $bs'a=\tt{0}s'a$.  Since $p$ ends in $\tt{1}$ or $\tt{2}$ and $s'$ is a palindromic suffix of $p$, we see that $s'$ must begin and end in either $\tt{1}$ or $\tt{2}$.  Since $\tt{02}$ is not a factor of $\Gam$, we see that $s'$ must begin and end in $\tt{1}$.  Write $s'=\tt{1}s''\tt{1}$, so that $bs'a=\tt{01}s''\tt{11}$.  Since every occurrence of $\tt{01}$ in $\Gam$ is followed by $\tt{1}$ or $\tt{2}$, while every occurrence of $\tt{11}$ in $\Gam$ is preceded by $\tt{0}$, we see that $s''$ begins in $\tt{1}$ or $\tt{2}$ and ends in $\tt{0}$.  But this contradicts the fact that $s'$ is a palindrome.  So we have $b\in\{\tt{1},\tt{2}\}$, which means that $\rev{r}$ is a suffix of $\gamma(b)$, and it follows that $S$ is a suffix of $P$. 
    
    Finally, since every occurrence of $S$ in $P$ must arise from an occurrence of $s'$ in $p'$, and $s'$ is unioccurrent in $p'$, we conclude that $S$ is unioccurrent in $P$.
\end{proof}

Theorem~\ref{Theorem:Ternary4PowerFree} follows immediately from Proposition~\ref{Proposition:Ternary4PowerFree} and Proposition~\ref{Proposition:TernaryRich}.

\section{Open Problems}
\label{Section:Conclusion}

The first problem that we mention is analogous to the one considered by Lietard and Rosenfeld~\cite{LietardRosenfeld2020}.  

\begin{problem}
    Determine the subsets of $\mathbb{Z}$ over which there is an infinite additive $4$-power-free rich word.
\end{problem}

We have shown that there is an infinite additive $4$-power-free rich word over $\{\tt{0},\tt{1},\tt{2}\}$, and that there is no such word over any subset of $\mathbb{Z}$ of size $2$. It is possible that there is an infinite additive $4$-power-free rich word over \emph{every} subset of $\mathbb{Z}$ of size $3$.  In fact, based on computational evidence, the word $\Gam_{\tt{a,b,c}}$ obtained from $\Gam$ by applying the morphism defined by $\tt{0}\mapsto \tt{a}$, $\tt{1}\mapsto\tt{b}$, and $\tt{2}\mapsto \tt{c}$, appears to be additive $4$-power-free for many choices of distinct integers $\tt{a}$, $\tt{b}$, and $\tt{c}$.  Perhaps ideas similar to those of Lietard and Rosenfeld~\cite{LietardRosenfeld2020} could be applied here.

It is known that every sufficiently long rich word (over \emph{any} alphabet) contains an ordinary square~\cite{PelantovaStarosta2013}, and hence an additive square.  However, a positive answer to the following problem seems plausible to us.

\begin{problem}
    Is there an infinite additive cube-free rich word over some finite subset of $\mathbb{Z}$?
\end{problem}

\section*{Acknowledgements}

When searching for a ternary morphism with a rich and additive $4$-power-free fixed point, we originally found the morphism $\delta$, and did not notice that it is the square of the morphism $\gamma$.  We thank Pascal Ochem for alerting us to this fact, which simplified the proof of Proposition~\ref{Proposition:TernaryRich}.


\begin{thebibliography}{10}

\bibitem{Andrade2024}
J.~Andrade, 
Avoiding additive powers in words, 
undergraduate Honour's thesis, Thompson Rivers University, 2024. Available at \url{https://tru.arcabc.ca/islandora/object/tru%3A6415}.

\bibitem{BalkovaPelantovaStarosta2011}
L'~Balkov\'a, E.~Pelantov\'a, and \v{S}.~Starosta,
Infinite words with finite defect,
{\em Adv. Appl. Math.} {\bf 47}(3) (2011), 562--574.

\bibitem{BaranwalShallit2019} A.~R.~Baranwal and J.~Shallit.
Repetitions in infinite palindrome-rich words,
in R.~Mercas and D.~Reidenbach, editors, {\em WORDS 2019}, Vol.~11682
of {\em Lecture Notes in Computer Science}, pp.~93--105, Springer-Verlag, 2019.

\bibitem{BlondinMasseEtAl2011}
A. Blondin Mass\'e, S. Brlek, S. Labb\'e, and L. Vuillon, 
Palindromic complexity of codings of rotations, 
{\em Theoret. Comput. Sci.} {\bf 412}(46) (2011), 6455--6463.

\bibitem{CassaigneCurrieSchaefferShallit2014} J.~Cassaigne, J.~D.~Currie, L.~Schaeffer, and J.~Shallit,
Avoiding three consecutive blocks of the same size and same sum,
{\em J. ACM} {\bf 61}(2) (2014), 1--17.

\bibitem{CurrieMolRampersad2020}
J.~D.~Currie, L.~Mol, and N.~Rampersad,
The repetition threshold for binary rich words,
{\em Discrete Math. Theoret. Comput. Sci.} {\bf 22}(1) (2020),
 article no.~DMTCS-22-1-6.  

\bibitem{CurrieMolRampersadShallit2021}
J.~D.~Currie, L.~Mol, N.~Rampersad, and J.~Shallit,
Extending Dekking's construction of an infinite binary word avoiding abelian 4-powers,
preprint, 2021. Available at
  \url{https://arxiv.org/abs/2111.07857}.

\bibitem{CurrieRampersad2012}
J.~D.~Currie and N.~Rampersad,
Fixed points avoiding Abelian $k$-powers,
{\em J. Comb. Theory Ser. A} {\bf 119}(5) (2012), 942--948.

\bibitem{Dekking1979} F.~M.~Dekking, 
Strongly non-repetitive sequences and progression-free sets,
{\em J. Comb. Theory Ser. A} {\bf 27}(2) (1979), 181--185.

\bibitem{DolcePelantova2022}
F.~Dolce and E.~Pelantov\'a,
On morphisms preserving palindromic richness,
{\em Fund. Inform.} {\bf 185}(1) (2022), 1--25.

\bibitem{DroubayJustinPirillo2001}
X.~Droubay, J.~Justin, and G.~Pirillo,
Episturmian words and some constructions of de Luca and Rauzy,
{\em Theoret. Comput. Sci.} {\bf 255}(1--2) (2001), 539--553.


\bibitem{Erdos1961}
P.~Erd\H{o}s,
Some unsolved problems,
{\em Magyar Tud. Akad. Mat. Kutat\'o Int. K\"ozl.} {\bf 6} (1961), 221--254.

\bibitem{GlenJustinWidmerZamboni2009}
A.~Glen, J.~Justin, S.~Widmer, and L.~Q.~Zamboni,
Palindromic richness,
{\em European J. Combin.} {\bf 30}(2) (2009), 510--531.

\bibitem{Keranen1992}
V.~Ker\"anen, 
Abelian squares are avoidable on 4 letters,
in W.~Kuich, editor, {\em Automata, Languages, and Programming 1992}, Vol.~623 of {\em Lecture Notes in Computer Science}, pp.~41--52, Springer, 1992.

\bibitem{LietardRosenfeld2020}
F.~Lietard and M.~Rosenfeld,
Avoidability of additive cubes over alphabets of four numbers,  
in N.~Jonoska and D.~Savchuk, editors, {\em Developments in Language Theory 2020}, Vol.~12086 of {\em Lecture Notes in Computer Science}, pp.~192--206, Springer-Verlag, 2020.

\bibitem{Lothaire1983}
M.~Lothaire,
{\em Combinatorics on Words},
Vol.~17 of {\em Encyclopedia of Mathematics and Its Applications}, Addison-Wesley, 1983.

\bibitem{Lothaire2002}
M.~Lothaire,
{\em Algebraic Combinatorics on Words},
Vol.~90 of {\em Encyclopedia of Mathematics and Its Applications}, Cambridge University Press, 2002.

\bibitem{Lothaire2005}
M.~Lothaire,
{\em Applied Combinatorics on Words},
Vol.~105 of {\em Encyclopedia of Mathematics and Its Applications}, Cambridge University Press, 2005.

\bibitem{PelantovaStarosta2013}
E.~Pelantov\'a and \v{S}.~Starosta, 
Languages invariant under more symmetries: Overlapping factors versus palindromic richness,
{\em Discrete Math.} {\bf 313}(21) (2013), 2432--2445.

\bibitem{PirilloVarricchio1994}
G.~Pirillo and S.~Varricchio,
On uniformly repetitive semigroups,
{\em Semigroup Forum} {\bf 49} (1994), 125--129.

\bibitem{Rao2015}
M.~Rao,
On some generalizations of abelian power avoidability,
{\em Theoret. Comput. Sci.} {\bf 601} (2015), 39--46.

\bibitem{RaoRosenfeld2018}
M.~Rao and M.~Rosenfeld,
Avoiding two consecutive blocks of same size and
same sum over $\mathbb{Z}^2$, 
{\em SIAM J. Discrete Math.}
{\bf 32}(4) (2018), 2381--2397.

\bibitem{RubinchikShur2018}
M.~Rubinchik and A.~M.~Shur,
EERTREE: An efficient data structure for processing palindromes in strings,
{\em European J. Combin.} {\bf 68} (2018), 249--265.

\bibitem{SchaefferShallit2016}
L.~Schaeffer and J.~Shallit,
Closed, palindromic, rich, privileged, trapezoidal, and balanced words in automatic sequences,
{\em Electron. J. Combin.} {\bf 23}(1) (2016), article no.~P1.25.  

\bibitem{Shallit2022}
J.~Shallit,
{\em The Logical Approach To Automatic Sequences: Exploring
  Combinatorics on Words with {\tt Walnut}}, Vol.~482 of {\em London Math. Soc.
  Lecture Note Series}, Cambridge University Press, 2022.

\bibitem{Thue1912}
A.~Thue,
\"Uber die gegenseitige Lage gleicher Teile gewisser Zeichenreihen,
{\em Norske vid. Selsk. Skr. Mat. Nat. Kl.} {\bf 1} (1912), 1--67.

\end{thebibliography}
\end{document}